\documentclass[11pt, leqno]{article}

\usepackage[dvips]{graphicx}

\usepackage{graphicx}
\usepackage{amssymb}
\usepackage{amsmath}
\usepackage{color}
\usepackage{times}

\usepackage[unicode,bookmarks,colorlinks]{hyperref}
\hypersetup{
    linkcolor=brickred,
}

\definecolor{mahogany}{cmyk}{0, 0.77, 0.87, 0}
\definecolor{salmon}{cmyk}{0, 0.53, 0.38, 0}
\definecolor{melon}{cmyk}{0, 0.46, 0.50, 0}
\definecolor{yellowgreen}{cmyk}{0.44, 0, 0.74, 0}
\definecolor{brickred}{cmyk}{0, 0.89, 0.94, 0.28}
\definecolor{OliveGreen}{cmyk}{0.64, 0, 0.95, 0.40}
\definecolor{RawSienna}{cmyk}{0, 0.72, 1.0, 0.45}
\definecolor{ZurichRed}{rgb}{1, 0, 0} 

\usepackage{amsmath,amstext,amssymb,amsopn,amsthm}
\usepackage{amsmath,amssymb,amsthm}
\usepackage[mathscr]{eucal}

\newtheorem{definition}{Definition}

\begin{document}

\newtheorem{lemma}[thm]{Lemma}
\newtheorem{remark}{Remark}
\newtheorem{proposition}{Proposition}
\newtheorem{theorem}{Theorem}[section]
\newtheorem{deff}[thm]{Definition}
\newtheorem{case}[thm]{Case}
\newtheorem{prop}[thm]{Proposition}
\newtheorem{example}{Example}

\newtheorem{corollary}{Corollary}

\numberwithin{equation}{section}
\numberwithin{definition}{section}
\numberwithin{corollary}{section}

\numberwithin{theorem}{section}

\numberwithin{remark}{section}
\numberwithin{example}{section}
\numberwithin{proposition}{section}

\newcommand{\gap}{\lambda_{2,D}^V-\lambda_{1,D}^V}
\newcommand{\gapR}{\lambda_{2,R}-\lambda_{1,R}}
\newcommand{\bD}{\mathrm{I\! D\!}}
\newcommand{\calD}{\mathcal{D}}
\newcommand{\calA}{\mathcal{A}}

\newcommand{\conjugate}[1]{\overline{#1}}
\newcommand{\abs}[1]{\left| #1 \right|}
\newcommand{\cl}[1]{\overline{#1}}
\newcommand{\expr}[1]{\left( #1 \right)}
\newcommand{\set}[1]{\left\{ #1 \right\}}

\newcommand{\calC}{\mathcal{C}}
\newcommand{\calE}{\mathcal{E}}
\newcommand{\calF}{\mathcal{F}}
\newcommand{\Rd}{\mathbb{R}^d}
\newcommand{\BR}{\mathcal{B}(\Rd)}
\newcommand{\R}{\mathbb{R}}
\newcommand{\al}{\alpha}
\newcommand{\RR}[1]{\mathbb{#1}}
\newcommand{\bR}{\mathrm{I\! R\!}}
\newcommand{\ga}{\gamma}
\newcommand{\om}{\omega}
\newcommand{\A}{\mathbb{A}}
\newcommand{\bH}{\mathbb{H}}
\newcommand{\B}{\mathbb{B}}

\newcommand{\Pro}{\mathbb{P}}
\newcommand\F{\mathcal{F}}
\newcommand\e{\varepsilon}
\def\H{\mathcal{H}}
\def\t{\tau}

\newcommand{\bb}[1]{\mathbb{#1}}
\newcommand{\bI}{\bb{I}}
\newcommand{\bN}{\bb{N}}

\newcommand{\uS}{\mathbb{S}}
\newcommand{\M}{{\mathcal{M}}}
\newcommand{\calB}{{\mathcal{B}}}

\newcommand{\W}{{\mathcal{W}}}

\newcommand{\m}{{\mathcal{m}}}

\newcommand {\mac}[1] { \mathbb{#1} }

\newcommand{\bC}{\Bbb C}

\newcommand{\ang}[1]{\left<#1\right>}  

\newcommand{\brak}[1]{\left(#1\right)}    
\newcommand{\crl}[1]{\left\{#1\right\}}   
\newcommand{\edg}[1]{\left[#1\right]}     

\newcommand{\E}[1]{{\rm E}\left[#1\right]}
\newcommand{\var}[1]{{\rm Var}\left(#1\right)}
\newcommand{\cov}[2]{{\rm Cov}\left(#1,#2\right)}
\newcommand{\N}[1]{||#1||}     
\newcommand{\bM}{\mathbb M}


\title{Martingale transforms and their projection operators on manifolds}
\author{Rodrigo Ba\~nuelos\thanks{Supported in part by NSF Grant \#0603701-DMS},\,  Fabrice Baudoin\footnote{Supported in part by NSF Grant 0907326--DMS}
\\Department of  Mathematics
\\Purdue University\\West Lafayette, 47906}
\maketitle

\begin{abstract}  
We prove the boundedness on $L^p$, $1<p<\infty$,  of operators on manifolds which arise by taking conditional expectation of transformations of stochastic integrals.  These operators include various classical operators such as second order Riesz transforms and operators of Laplace transform-type.  
\end{abstract}
\tableofcontents

\section{Introduction}

The classical martingale inequalities of Burkholder and Gundy \cite{BurGun} play a fundamental role in many areas of probability and its applications.  These inequalities have their roots in the celebrated 1966 martingale transforms inequality of Burkholder \cite{Bur1}. In 1984 \cite{Bur2}, Burkholder obtained the sharp constants in his 1966 martingale inequalities.  In recent years, these sharp inequalities have  had many applications to the study of  basic singular integrals and Fourier multipliers on  Euclidean space $\bR^d$ with the Lebesgue measure (the ordinary Laplacian case) and with Gaussian measure (the Ornstein--Uhlenbeck case). For an account of some of this literature we refer the reader to the overview article  \cite{Ban}.  The purpose of this paper is to show that these martingale transform techniques apply to wide range of operators on manifolds,  including multipliers arising from Schr\"odinger operators and Riesz transforms on Lie groups.   For the Laplacian, and another self-adjoint diffusions without a zero order term, Burkholder's sharp  inequalities can be applied and our $L^p$ bounds are exactly as those on $\bR^d$.  However, in the case of Schr\"odinger operators the sharp Burkholder inequalities do not apply in any direct way and in this case we obtain our results by proving a version of the Burkholder-Gundy inequalities with an exponential weight which  does not produce best constants.  Our results here also correct a gap contained in \cite{Li1}, \cite{Li2}, \cite{Li3}; see Remark \ref{correction} below. 

To introduce our operators and state our results,  let  $\mathbb{M}$ be a smooth manifold endowed with a smooth measure $\mu$. Let $X_1,\cdots, X_d$ be locally Lipschitz vector fields defined on $\mathbb{M}$.   We consider the Schr\"odinger operator,
\[
L=-\frac{1}{2} \sum_{i=1}^d X_i^* X_i+V,
\]
where $X_i^*$ denotes the formal adjoint of $X_i$ with respect to $\mu$ and where $V:\mathbb{M} \rightarrow \mathbb{R}$ is a non-positive smooth function, often referred to it in this paper as a {\it potential}.  We assume that $L$ is essentially self-adjoint with respect to $\mu$ on the space $\mathcal{C}^\infty_0(\mathbb{M})$ of smooth and compactly supported functions. We denote by $(P_t)_{t \ge 0}$ the heat semigroup with generator  $L$. We can write $L=\frac{1}{2} \sum_{i=1}^d X^2_i+X_0+V$, for some locally Lipschitz  vector field $X_0$. 

 Let $A_{ij} : [0,+\infty) \times \mathbb{M} \rightarrow \mathbb{R}$, $1\le i,j \le d$ be bounded and smooth real valued functions and  let $A(t,x)=\left(A_{ij}\right)$ be the $d\times d$ matrix with $A_{ij}$ entries.  Set 
$$
\|A\|=\| |A(t,x)|\|_{L^{\infty} ([0,+\infty) \times \mathbb{M})},
$$
 where $|A(t,x)|$ is the usual quadratic norm of the $d\times d$ matrix $A(t,x)$.  That is,
 $$
 |A(t, x)|=\sup\{|A(t, x)\xi|;  \xi\in \bR^d, |\xi|=1\}. 
 $$
  Our goal in this paper is to study  the continuity on $L_\mu^p(\mathbb{M})$ for $1<p<\infty$ of the operator
\begin{equation}\label{operators}
\mathcal{S}_A f=\sum_{i,j=1}^d\int_{0}^{\infty}   P_t  X_i^*  A_{ij} (t) X_j P_t fdt
\end{equation}
and to obtain precise information on the size of their  norms. The following two remarks shed some light on the structure of these operators in two important special cases. 
\begin{remark}
 If $L=-\frac{1}{2} \sum_{i=1}^d X_i^* X_i+V$ is subelliptic, then the semigroup $P_t$ admits a  smooth symmetric kernel $p(t,x,y)$ (see \cite{baudoin},  \cite{JSC2} ) and in that case, it is easily seen that
\[
\mathcal{S}_A f (x)=\int_{\mathbb{M}} K(x,z) f(z) d\mu(z),
\]
with
\[
K(x,z)=\sum_{i,j=1}^d \int_0^{+\infty} \int_{\mathbb{M}} A_{ij}(t,y) X^y_i p(t,y,x) X^y_j p(t,y,z) d\mu(y) dt.
\]
\end{remark}

\begin{remark}
Let $(\mathbb{M}, g)$ be a complete Riemannian manifold. In this framework the Laplace-Beltrami operator $\Delta$ is essentially self-adjoint on $\mathcal{C}^\infty_0(\mathbb{M})$. We may then consider the Schr\"odinger operator
\[
L=-\frac{1}{2} \Delta+V
\]
where, as above, $V$ is a non-positive smooth potential. In this case, the operator $\mathcal{S}_A$ can be written as 
\[
\mathcal{S}_A f=\int_{0}^{\infty}   P_t  \mathbf{div} (A(t) \nabla P_t f) dt.
\]
\end{remark}

For $1<p<\infty$ let $p^*$  denote the maximum of $p$ and $q$, where $\frac{1}{p}+\frac{1}{q}=1$.   Thus $p^* = \max\{p,\frac{p}{p-1}\}$ and 
\begin{equation}\label{p^*}
p^*-1=\begin{cases} \frac{1}{p-1},  \hskip4mm  1<p\leq 2,\\
p-1 , \hskip3mm  2\leq p <\infty.
\end{cases}
\end{equation}

The following is the main result of this paper. 

\begin{theorem}\label{main}  For any $1<p<\infty$ there is a constant $C_p$ depending only on $p$ such that for every $f \in L_\mu^p(\mathbb{M})$,
\begin{equation}\label{Schro}
\| \mathcal{S}_A f \| \le C_p \|A \| \| f \|_p.
\end{equation}
If the potential $V\equiv 0$, then 
\begin{equation}\label{V=0} 
\| \mathcal{S}_A f \|_p \le (p^*-1) \|A \| \| f \|_p,
\end{equation}
 and this bound is sharp. 
\end{theorem}

As we shall see below, these operators include the multipliers of Laplace transform-type and second order Riesz transforms on Lie groups of compact type.  The fact that the bound in \eqref{V=0} is best possible follows from the fact that the best constant in the $L^p$ inequality for second order Riesz transforms $R_jR_k$ on $\bR^d$ is $p^*-1$; see \S \ref{LieRiesz} below and \cite{GeiSmiSak} and \cite{BanOse}. 

The paper is organized as follows.  In \S2, we recall various versions of Burkholder's sharp inequalities and prove a version of the Burkholder-Gundy inequality needed for \eqref{Schro}.  The proof of Theorem \ref{main} is given in \S3 where we also  give an explicit bound for the constant in \eqref{Schro}.  In \S4, we give some concrete examples of our operators. 
\section{Martingale  inequalities}

In this section we recall the sharp martingale inequalities of Burkholder and prove a version of the Burkholder-Gundy inequalities used in the proof of \eqref{Schro} in Theorem \ref{main}. 

 Let  $f=\{f_n, n\geq 0\}$ be a martingale with different sequence $df=\{df_k, k\geq 0\}$, where $df_k=f_k-f_{k-1}$ for $k\geq 1$ and $df_0=f_0$.  Given a predictable sequence of random variables $\{v_k, k\geq 0\}$ uniformly bounded for all $k$ the martingale difference sequence $\{v_kdf_k, k\geq 0\}$ generates a new martingale called the {\it martingale transform} of $f$ and  denoted by $v\ast f$.   We set $\|f\|_p=\sup_{n\geq 0}\|f_n\|_p$ for $0<p<\infty$. Burkholder's 1966 result  \cite{Bur1} asserts that the operator $f\to v\ast f=g$ is bounded on $L^p$ for all $1<p<\infty$. In his 1984 seminal paper \cite{Bur2} Burkholder determined the norm of this operator by proving the following result. 
\begin{theorem}\label{Burkholder1} Let $f=\{f_n, n\geq 0\}$ be a martingale  and let $g=v\ast f$ be its martingale transform by the predictable sequence $v=\{v_k, k\geq 0\}$ with $v_k$ taking values in $[-1, 1]$ for all $k$.  Then 
\begin{equation}\label{bur2}
\|g\|_p\leq (p^*-1)\|f\|_p, \qquad  1<p<\infty,
\end{equation}
and the constant $p^*-1$ is best possible. 
\end{theorem}

%

In \cite{Cho}, K.P. Choi used the techniques of Burkholder to identify the best constant in the martingale transforms where the predictable sequence $v$ takes values in $[0, 1]$ instead of $[-1, 1]$. While Choi's  constant is not as explicit as the $p^*-1$ constant of Burkholder, one does have a lot of information on it. 

 \begin{theorem}\label{Choi1} Let $f=\{f_n, n\geq 0\}$ be a real-valued martingale and let $g=v\ast f$ be its martingale by a predictable sequence $v=\{v_k, k\geq 0\}$ with values in $[0,1]$ for all $k$.  Then 
\begin{equation}\label{choi1}
\|g\|_p\leq c_p\|f\|_p, \qquad  1<p<\infty,
\end{equation}
with the best constant $c_p$ satisfying
$$
c_p=\frac{p}{2}+ \frac{1}{2}\log\left(\frac{1+e^{-2}}{2}\right) +\frac{\alpha_2}{p}+\cdots
$$
where 
$$\alpha_2=\left[\log\left(\frac{1+e^{-2}}{2}\right)\right]^2+\frac{1}{2}\log\left(\frac{1+e^{-2}}{2}\right)-2\left(\frac{e^{-2}}{1+e^{-2}}\right)^{2}. $$
\end{theorem}

%
%
Motivated by Theorems \ref{Burkholder1} and \ref{Choi1} the following definition was introduced in \cite{BanOse}.

\begin{definition}\label{defC}
Let $-\infty<b<B<\infty$ and $1<p<\infty$ be given and fixed. We define $C_{p,b,B}$ as the least positive number $C$ such that for any real-valued martingale $f$ and for any transform $g=v\ast f$ of $f$ by a predictable sequence $v=\{v_k, k\geq 0\}$ with values in $[b,B]$,  we have
\begin{equation}\label{martin}
 ||g||_p\leq C||f||_p.
\end{equation}
\end{definition}

Thus, for example, $C_{p,-a,a}=a(p^*-1)$ by Burkholder's Theorem 
 \ref{Burkholder1} and $C_{p,0,a}=a\,c_p$ by Choi's Theorem \ref{Choi1}.  It is also the case that for any $b, B$ as above, 
 
 \begin{equation}
 \max\left\{\left(\frac{B-b}{2}\right)(p^*-1), \, \max\{|B|, |b|\}\right\} \leq C_{p,b,B}\leq \max\{B, |b|\}(p^*-1)
  \end{equation}

For the applications to the above operators we will need versions of these inequalities for continuous-time martingales. Suppose that $(\Omega,\mathcal{F},\mathbb{P})$ is a complete probability space, filtered by $(\mathcal{F}_t)_{t\geq 0}$, a nondecreasing and right-continuous family of sub-$\sigma$-fields of $\mathcal{F}$. Assume, as usual, that $\F_0$ contains all the events of probability $0$. Let $X$, $Y$ be adapted, real valued martingales which have right-continuous paths with left-limits (r.c.l.l.).  Denote by $[X,X]$ the quadratic variation process of $X$: we refer the reader to Dellacherie and Meyer \cite{DelMey} for details. 
Following  \cite{BanWan} and  \cite{Wan}, we say that $Y$ is differentially subordinate to $X$ if the process if $|Y_0| \leq |X_0|$ and $([X,X]_t-[Y,Y]_t)_{t\geq 0}$ is nondecreasing and nonnegative as a function of $t$. We have the following extension of the Burkholder inequalities  proved in \cite{BanWan} for continuous-path martingales and in \cite{Wan} in the general case. Set $\|X\|_p=\sup_{t\geq 0}\|X_t\|$, $0<p<\infty$.
\begin{theorem}\label{symm}
If $Y$ is differentially subordinate to $X$, then
\begin{equation}\label{burkholder}
\|Y\|_p\leq (p^*-1)\|X\|_p, \qquad 1<p<\infty,
\end{equation}
and the inequality is sharp.  
\end{theorem}

The case of non-symmetric multipliers is covered by the following result proved in \cite{BanOse}.
\begin{theorem}\label{non-symm}
\small
Suppose $-\infty<b<B<\infty$ and  $X_t$, $Y_t$ are two real valued martingales with which have right-continuous paths with left-limits with $|Y_0|\leq |X_0|$ and which satisfy 
\begin{equation}\label{banoseestiamte}
\left[\frac{B-b}{2}X,\frac{B-b}{2}X\right]_t-\left[Y-\frac{b+B}{2}X,Y-\frac{b+B}{2}X\right]_t\geq 0
\end{equation}
and nondecreasing for all $t\geq 0$ (differential subordination). Then 
\begin{equation*}
 ||Y||_p\leq C_{p,b,B}||X||_p, \quad 1<p<\infty,
\end{equation*}
and the inequality is sharp.
\end{theorem}
Note that when $B=a>0$ and $b=-a$, we have the case of Theorem \ref{symm}.  As we shall see, \eqref{burkholder} will give the bound in \eqref{V=0} and \eqref{banoseestiamte} will give some extensions. 

For the general case when $V\not=0$ (for the  inequality \eqref{Schro}) none of the above results apply and we shall need a variation of the Burkholder-Gundy inequalities.  While this bound is not sharp, it applies to a wide class of processes and can even be used to study operators on manifolds acting on forms.  We shall comment more on this a little later. 
We start by  recalling the following domination inequality due to Lenglart \cite{Len} (See also Revuz-Yor \cite{RevYor}, p.162-163).

\begin{proposition}(Lenglart)
Let $(N_t)_{t \ge 0}$ be a positive adapted right-continuous process and $(A_t)_{t \ge 0}$ be an increasing process. Assume that for every bounded stopping time $\tau$,
\[
\mathbb{E} (N_\tau) \le \mathbb{E} (A_\tau).
\]
Then, for every $k \in (0,1)$,
\[
\mathbb{E} \left( \left(\sup_{0 \le t \le T} N_t \right)^k \right) \le \frac{2-k}{1-k} \mathbb{E} \left( A_T^k\right).
\]
\end{proposition}

We shall use this lemma to prove the following 

\begin{theorem}\label{BDG}
 
Let $T>0$ and $(M_t)_{ 0 \le t \le T}$ be a continuous local martingale. 
Consider the process
\[
Z_t=e^{\int_0^t V_s ds} \int_0^t e^{-\int_0^s V_u du} dM_s,
\]
where $(V_t)_{ 0 \le t \le T}$ is a non positive  adapted and continuous process.  
For every $0<p<\infty$, there is a universal constant $C_p$, independent of $T$, $(M_t)_{ 0 \le t \le T}$  and $(V_t)_{ 0 \le t \le T}$ such that
\[
 \mathbb{E}\left(\left(\sup_{0 \le t \le T} |Z_t|\right)^p \right) \le C_p  \mathbb{E} \left( \langle M \rangle_T^{\frac{p}{2} } \right).
 \]
\end{theorem}

\begin{proof}
By stopping it is enough to prove the result for bounded $M$.
Let $q \ge 2$. From It\^o's formula we have
\[
dZ_t=Z_t V_t dt+ dM_t
\]
and 
\begin{align*}
d |Z_t|^q & =q|Z_t|^{q-1} \mathbf{sgn}(Z_t) dZ_t +\frac{1}{2} q(q-1) | Z_t |^{q-2} d<M>_t \\
 & =q|Z_t|^q V_t dt+q  \mathbf{sgn}(Z_t)|Z_t|^{q-1}dM_t+ \frac{1}{2} q(q-1) | Z_t |^{q-2}d<M>_t.
\end{align*}
Since $V_t \le 0$, as a consequence of the Doob's optional sampling theorem, we get that for every bounded stopping time $\tau$,
\[
\mathbb{E} \left(  |Z_\tau|^q  \right) \le  \frac{1}{2} q(q-1) \mathbb{E} \left( \int_0^\tau | Z_t |^{q-2} d<M>_t \right).
\]
From the Lenglart's domination inequality, we deduce then that for every $k \in (0,1)$,
\[
\mathbb{E} \left( \left(\sup_{0 \le t \le T}  |Z_t|^q  \right)^k \right) \le \frac{2-k}{1-k}  \left( \frac{1}{2} q(q-1)\right)^k  \mathbb{E} \left(\left( \int_0^T | Z_t |^{q-2}  d<M>_t\right)^k \right).
\]
We now bound
\begin{align*}
 \mathbb{E} \left(\left( \int_0^T | Z_t |^{q-2} d<M>_t\right)^k \right) & \le  \mathbb{E} \left(\left(\sup_{0 \le t \le T} |Z_t| \right)^{k(q-2)}\left( \int_0^T d<M>_t\right)^k \right) \\
  & \le  \mathbb{E} \left(\left(\sup_{0 \le t \le T} |Z_t| \right)^{kq} \right)^{1-\frac{2}{q}} \mathbb{E}  \left(<M>_T^{\frac{kq}{2}} \right)^{\frac{2}{q}}.
\end{align*}
As a consequence, we obtain
\begin{align*}
\mathbb{E} \left( \left(\sup_{0 \le t \le T}  |Z_t|^q  \right)^k \right) \le \frac{2-k}{1-k}  \left( \frac{1}{2} q(q-1)\right)^k   \mathbb{E} \left(\left(\sup_{0 \le t \le T} |Z_t| \right)^{kq} \right)^{1-\frac{2}{q}} \mathbb{E}  \left(<M>_T^{\frac{kq}{2}} \right)^{\frac{2}{q}}.
\end{align*}
Letting  $p=qk$ yields the claimed result.
\end{proof}

We should note here that in the above proof the fixed time $T$ can be replaced by a stopping time.  More interesting and useful for applications is the fact that this result admits the following  multidimensional generalization whose proof we leave to the interested reader.   

\begin{theorem}\label{GeneralB-G}
 Let $T>0$ and $(M_t)_{ 0 \le t \le T}$ be an $\mathbb{R}^d$ valued continuous local martingale. 
Consider the solution of the matrix equation
\[
d\mathcal{M}_t =\mathcal{V}_t \mathcal{M}_t dt, \quad \mathcal{M}_0=\mathbf{Id},
\]
where $(\mathcal{V}_t)_{0 \le t \le T}$ is an adapted and continuous process taking values in the set of symmetric and non positive $d \times d$ matrices. Consider the process
\[
Z_t=\mathcal{M}_t \int_0^t \mathcal{M}^{-1}_sdM_s.
\]
 For every $0<p<\infty$, there is a universal constant $C_p$, that is independent from $d$,  $T$, $(M_t)_{ 0 \le t \le T}$  and $(\mathcal{V}_t)_{ 0 \le t \le T}$ such that
\begin{align}\label{multi}
 \mathbb{E}\left(\left(\sup_{0 \le t \le T} \| Z_t \| \right)^p \right) \le C_p  \mathbb{E} \left( \| \langle M \rangle_T\|^{\frac{p}{2} } \right).
 \end{align}
 \end{theorem}
 
 \begin{remark}\label{correction}
 Theorem \ref{GeneralB-G} may be used to correct a gap contained in the papers \cite{Li1}, \cite{Li2}, \cite{Li3} of  X. D. Li. In these papers the author considers quantities like  $\mathcal{M}_t \int_0^t \mathcal{M}^{-1}_sdM_s$ but writes them  as $\ \int_0^t \mathcal{M}_t \mathcal{M}^{-1}_sdM_s$ to give $L^p$ estimates using the classical Burkholder-Gundy  inequality or even explicit expressions for the constants using the Burkholder bounds of Theorem \ref{symm}.  This, however, is not possible due to the non-adaptedness of the process  $\mathcal{M}_t \mathcal{M}^{-1}_s$ which prevents us from bringing the $\mathcal{M}_t$ inside the integral.  Therefore, in those papers, many of the explicit constants given there involving $(p^*-1)$ need to be replaced with  less precise universal constants $C_p$ depending only on $p$ as in (\ref{multi}).
 \end{remark}

\section{Proof of Theorem \ref{main} and some refinements}


Consider the Schr\"odinger operator  $L=\frac{1}{2} \sum_{i=1}^d X^2_i+X_0+V$. The diffusion $(Y_t)_{t \ge 0}$ with generator $\frac{1}{2} \sum_{i=1}^d X^2_i+X_0$  can then be constructed via the Stratonovitch stochastic differential equation
\[
dY_t =X_0(Y_t) dt + \sum_{i=1}^d X_i(Y_t) \circ dB^i_t,
\]
which we assume non--explosive. In the sequel we shall denote by $(Y_t)_{t \ge 0}$ the solution of this equation started with an initial \textit{distribution} $\mu$.  We will use $E_x$ to denote the expectation associated with the process $Y$starting at $x$ and $E$ to denote the expectation of the process starting with $\mu$.  (See \cite{Ban} for some literature related to the possible intricacies associated to the process starting with the possibly infinite measure $\mu$.) 

Let us recall that by the Feynman-Kac formula, the semigroup $P_t$ acting on $f \in C_0^\infty (\mathbb{M})$ can be written as

$$
P_tf(x)=E_x\left( e^{\int_0^t V(Y_s)ds} f(Y_t)\right). 
$$

%
%
%
%
Let us fix $T>0$ in what follows. We first consider the relevant martingales associated to our operators. 
We have 

\begin{lemma}\label{martingale}
Let $f \in C_0^\infty (\mathbb{M})$. The process $\left( e^{\int_0^t V(Y_s)ds} (P_{T-t} f)(Y_t) \right)_{0 \le t \le T}$ is a martingale and we have
\[
e^{\int_0^T V(Y_s)ds} f(Y_T)=(P_{T} f)(Y_0)+ \sum_{i=1}^d \int_0^T e^{\int_0^t V(Y_s)ds} (X_i P_{T-t} f)(Y_t)dB^i_t .
\]
\end{lemma}

The proof of the lemma is clear.  For $0\leq t\leq T$, set 
$$
Z_t^1=e^{\int_0^t V(Y_s)ds}, \quad Z_t^2=(P_{T-t} f)(Y_t) 
$$
and apply the It\^o formula.  

The next expression provides the probabilistic connection to our operators. For $f \in C_0^\infty (\mathbb{M})$, $0<T<\infty$, set
\[
\mathcal{S}^T_A f (x) =\mathbb{E} \left(  e^{\int_0^T V(Y_s)ds} \sum_{i,j=1}^d \int_0^T e^{-\int_0^t V(Y_s)ds} A_{ij}(T-t,Y_t) (X_jP_{T-t} f)(Y_t)dB^i_t  \mid Y_T=x \right).
\]
If $g$ is another function in  $C_0^\infty (\mathbb{M})$, the It\^o isometry and the lemma give 
\begin{eqnarray*}
\int_\bM (\mathcal{S}^T_A f ) g d\mu &=&\mathbb{E} \left(  e^{\int_0^T V(Y_s)ds} \sum_{i,j=1}^d \int_0^T e^{-\int_0^t V(Y_s)ds} A_{ij}(T-t,Y_t) (X_jP_{T-t} f)(Y_t)dB^i_t  g(Y_T) \right) \\ 
 &=& \sum_{i,j=1}^d \mathbb{E} \left(  e^{\int_0^T V(Y_s)ds}  g(X_T) \int_0^T e^{-\int_0^t V(Y_s)ds} A_{ij}(T-t,Y_t) (X_jP_{T-t} f)(Y_t)dB^i_t   \right) \\
  &= &\sum_{i,j=1}^d \mathbb{E} \left(  \int_0^T  A_{ij}(T-t,Y_t)(X_iP_{T-t} g)(Y_t) (X_jP_{T-t} f)(Y_t)dt  \right) \\
 &=&\sum_{i,j=1}^d \int_0^{T} \int_{\mathbb{M}}  A_{ij}(t)(X_iP_{t} g) (X_jP_{t} f) d\mu dt.
\end{eqnarray*}
Therefore, we have
\begin{equation}\label{S_A^T}
\mathcal{S}^T_A f =\sum_{i,j=1}^d \int_0^T (P_t X_i^* A_{ij}(t) X_j P_t f) dt.
\end{equation}

\begin{lemma}\label{mainT}  For any $1<p<\infty$, there is a constant $C_p$ depending only on $p$ such that for every $f \in L_\mu^p(\mathbb{M})$,
\begin{equation}\label{SchroT}
\| \mathcal{S}_A^T f \| \le C_p \|A \| \| f \|_p.
\end{equation}
If the potential $V\equiv 0$, then 
\begin{equation}\label{V=0-T} 
\| \mathcal{S}_A^T f \| \le (p^*-1) \|A \| \| f \|_p. 
\end{equation}
In particular these constants do not depend on $T$. 
\end{lemma}

\begin{proof}  We first observe that if $V\equiv 0$ then the martingale 
$$
\sum_{i,j=1}^d \int_0^t A_{ij}(T-t,Y_t) (X_jP_{T-t} f)(Y_t)dB^i_t 
$$
is differentially subordinate to $\|A\| P_{T-t}f(Y_t).$  It Follows from the contraction of the conditional expectation on $L^p$, $1<p<\infty$, and Theorem \ref{symm}, that 
$$
\| \mathcal{S}_A^T f \| \le (p^*-1) \|A \| \| f \|_p,
$$
 which is the estimate in \eqref{V=0-T}. 

We now deal with the case of $V\not=0$. We first prove an $L^p$ estimates for 
$$\int_0^T \sum_{i=1}^d (X_iP_{T-t} f)^2 (Y_t)dt$$
 which is the quadratic variation of the martingale appearing in Lemma \ref{martingale}.  Using It\^o's formula with $0\leq t\leq T$ for $(P_{T-t}f)(Y_t)^2$, we obtain
\begin{eqnarray*}
f(Y_T)^2=P_Tf(Y_0)^2&+&\int_0^T \left( \frac{1}{2} \sum_{i=1}^d X_i^2 +X_0 +\frac{\partial}{\partial t} \right) (P_{T-t}f)^2 (Y_t) dt\\
& +&\sum_{i=1}^d \int_0^T X_i (P_{T-t}f)^2(Y_t)dB^i_t. 
\end{eqnarray*}
Therefore
\[
\mathbb{E} \left(\int_0^T \left( \frac{1}{2} \sum_{i=1}^d X_i^2 +X_0 +\frac{\partial}{\partial t} \right) (P_{T-t}f)^2 (Y_t)  dt \right) \le \| f\|_2^2.
\]
We now compute
\begin{align*}
 \left(  \frac{1}{2} \sum_{i=1}^d X_i^2 +X_0 +\frac{\partial}{\partial t} \right) (P_{T-t}f)^2&  = 2 (P_{T-t}f) \left( \frac{1}{2} \sum_{i=1}^d X_i^2 +X_0\right)P_{T-t}f  +\sum_{i=1}^d (X_iP_{T-t} f)^2  \\
 &-2 (P_{T-t}f)\left( \frac{1}{2} \sum_{i=1}^d X_i^2 +X_0 +V\right) P_{T-t}f \\
  &= \sum_{i=1}^d (X_iP_{T-t} f)^2 -2 V (P_{T-t}f)^2 \\
   & \ge \sum_{i=1}^d (X_iP_{T-t} f)^2.
\end{align*}
This implies that 
\begin{align*}
 \mathbb{E} \left(  \int_0^T \sum_{i=1}^d (X_iP_{T-t} f)^2(Y_t) dt \right) \le\mathbb{E} \left(\int_0^T \left( \frac{1}{2} \sum_{i=1}^d X_i^2 +X_0 +\frac{\partial}{\partial t} \right) (P_{T-t}f)^2 (Y_t)  dt \right)  \le \| f\|_2^2.
\end{align*}
Combining this with Theorem \ref{BDG} and again with the fact that the conditional expectation is a contraction on $L^2$, we obtain 
\[
\| \mathcal{S}^T_A f \|_2 \le \|A \| \| f \|_2, 
\]
which proves the result for $p=2$. 

Assume now that  $2<p<\infty$. The above computations show that $(P_{T-t}f)^2 (Y_t)$ is a sub-martingale. From the Lenglart-L\'epingle-Pratelli estimate  and Doob's maximal inequality, we deduce that 
\begin{eqnarray*}
\mathbb{E} \left( \left( \int_0^T \left(  \frac{1}{2} \sum_{i=1}^d X_i^2 +X_0 +\frac{\partial}{\partial t} \right) (P_{T-t}f)^2 (Y_t) dt \right)^{\frac{p}{2}} \right)& \le& p^{p/2}  \mathbb{E} \left( \sup_{0 \le t \le T} \left((P_{T-t}f)^2 (Y_t) \right)^{p/2} \right) \\
 & \le&  p^{p/2} \left( \frac{p}{p-2} \right)^{p/2} \mathbb{E}(f (Y_T)^p ) \\
 & \le & p^{p/2} \left( \frac{p}{p-2} \right)^{p/2}\| f \|_p^p.
\end{eqnarray*}
We conclude
\begin{eqnarray*}
 \mathbb{E} \left( \left( \int_0^T   \sum_{i=1}^d  (X_i P_{T-t} f)(Y_t)^2 dt \right)^{\frac{p}{2}} \right)
 & \le& \mathbb{E} \left( \left( \int_0^T \left( \frac{1}{2} \sum_{i=1}^d X_i^2 +X_0 +\frac{\partial}{\partial t} \right) (P_{T-t}f)^2 (Y_t) dt \right)^{\frac{p}{2}} \right) \\
  &  \le&  p^{p/2} \left( \frac{p}{p-2} \right)^{p/2}\| f \|_p^p
\end{eqnarray*}
Combining this with Theorem \ref{BDG}, proves the result in the range $2<p<\infty$. Finally, the adjoint of $\mathcal{S}^T_A$ acting on $L_\mu^p(\mathbb{M})$ is $\mathcal{S}^T_{A^*}$ acting on $L_\mu^q(\mathbb{M})$, where $\frac{1}{p}+\frac{1}{q}=1$. The result is then obtained by duality.
\end{proof}

%

\begin{proof}[Proof of Theorem \ref{main}]
It remains to show that we can let $T\to \infty$ in Lemma \ref{mainT}. To  see this observe that by applying the lemma with the matrix 
\[
A_{ij}(t,x)=\frac{(X_iP_t f )(x)(X_j P_t g)(x)}{\left(\sum_{i=1}^d (X_iP_tf)^2 (x)\right)^{1/2}
\left(\sum_{i=1}^d (X_iP_tg)^2(x)\right)^{1/2} }
\]
 we obtain the uniform bound
\begin{equation}
\int_0^{T} \int_{\mathbb{M}}  \left(\sum_{i=1}^d (X_iP_tf)^2 (x)\right)^{1/2}\left( \sum_{i=1}^d (X_iP_tg)^2(x)\right)^{1/2}d\mu dt \le a_p \| f \|_p \| g \|_q
\end{equation}
with $\frac{1}{p}+\frac{1}{q}=1$, where $a_p=C_p$ in the case of $V\not=0$ and otherwise it is $p^*-1$.  Since 
\begin{align*}
\int_\bM (\mathcal{S}^T_A f ) g d\mu=\sum_{i,j=1}^d \int_0^{T} \int_{\mathbb{M}}  A_{ij}(t)(X_iP_{t} g) (X_jP_{t} f) d\mu dt,
\end{align*}
we deduce that
\begin{align*}
\lim _{T \to \infty} \int_\bM (\mathcal{S}^T_A f ) g d\mu & =\sum_{i,j=1}^d \int_0^{+\infty} \int_{\mathbb{M}}  A_{ij}(t)(X_iP_{t} g) (X_jP_{t} f) d\mu dt  \\
 & =  \int_\bM (\mathcal{S}_A f ) g d\mu 
\end{align*}
and moreover that
\[
\left|  \int_\bM (\mathcal{S}_A f ) g d\mu \right| \le a_p \| A \|  \| f \|_p \| g \|_q.
\]
This shows that the same bounds in Lemma \ref{mainT} hold with $\mathcal{S}_A^T$ replaced by $\mathcal{S}_A$. This completes the proof of Theorem \ref{main}. 
\end{proof}

The following corollary of the above proof generalizes to manifolds the key estimate of Nazarov and Volberg \cite{NazVol} and Dragi\v cevi\'c and Volberg \cite{DraVol1}, \cite{DraVol2}, for the Laplacian in $\bR^d$; see \cite[Corollary 3.9.1]{Ban}. 

\begin{corollary}\label{NazVol-manifolds}
For all $f, g\in C_0^\infty (\mathbb{M})$, 
$1< p<\infty$,   
\begin{equation}
\int_0^{\infty} \int_{\mathbb{M}}\left(\sum_{i=1}^d (X_iP_tf)^2 (x)\right)^{1/2}
\left(\sum_{i=1}^d (X_iP_tg)^2(x)\right)^{1/2}d\mu dt \le a_p \| f \|_p \| g \|_q,
\end{equation}
with $\frac{1}{p}+\frac{1}{q}=1$, where $a_p=C_p$ as in \eqref{SchroT}, if $V\not=0$  and  $a_p=p^*-1$, if $V\equiv 0$. 
\end{corollary}

We now state a result that follows from Theorem \ref{non-symm} when we have some additional information on the matrix $A=\left(A_{ij}\right)$.  Again, this is exactly as on $\bR^d$; see \cite{BanOse}. 

\begin{theorem}\label{non-symm-S_A}  Suppose $V\equiv 0$.  Let $A=\left(A_{ij}\right)$ be symmetric and suppose that there are universal constants  $-\infty<b<B<\infty$  such that for all $(t, x)\in [0, \infty)\times M$, 
$$b|\xi|^2\leq \sum_{i,j=1}^d A_{ij}(t, x)\xi_i\xi_j\leq B|\xi|^2,$$
for all $\xi\in \bR^d$.  Then for all $1<p<\infty$ we have 
\begin{equation}\label{non-symm-T}
\|\mathcal{S}_A f\|_p\leq C_{p,b,B} \|f\|_p,
\end{equation}
where $C_{p,b,B}$ is the constant in Theorem \ref{non-symm}. 
\end{theorem}
\begin{proof} As above, this result will follow if we can prove the same result for the operator 
\[
\mathcal{S}^T_A f (x) =\mathbb{E}\left(\sum_{i,j=1}^d \int_0^T A_{ij}(T-t,Y_t) (X_jP_{T-t} f)(Y_t)dB^i_t  \mid Y_T=x \right).
\] 
Consider the two martingales 
$$
Y_t=\sum_{i,j=1}^d \int_0^T A_{ij}(T-t,Y_t) (X_jP_{T-t} f)(Y_t)dB^i_t, \quad 
X_t=\sum_{i,j=1}^d \int_0^T (X_jP_{T-t} f)(Y_t)dB^i_t.
$$
It is simple to verify (see \cite{BanOse}) that under our assumptions on the matrix $A$, these martingales satisfy the hypothesis of Theorem \ref{non-symm}.  From this and the contraction of the conditional expectation on $L^p$ we obtain \eqref{non-symm-T} for the operators $\mathcal{S}^T_A$. 
\end{proof}

\subsection{An estimate of the constant in \eqref{Schro}}

The above approach based on Theorem \ref{BDG} to prove \eqref{Schro} is general but does not lead to explicit constants with useful information.  We present below an alternative argument which provides explicit constants with some additional information. We use the notations introduced in the previous section.

\begin{proposition} Suppose that $V\not=0$ and that 
$1<p<\infty$.  Then 
\[
\left\| \sum_{i=1}^d  \int_0^T (X_iP_{T-t}f)(Y_t)dB^i_t  \right\|_p  \le \frac{2^{2-\frac{1}{p} }p^2}{p-1} \| f \|_p
\]
\end{proposition}

\begin{proof}
From It\^o's formula applied to $(P_{T-t} f) (Y_t)$, we find
\[
(P_{T-t} f) (Y_t) =P_Tf (Y_0) +\int_0^t \left( \frac{\partial}{\partial s} +\frac{1}{2} \sum_{i=1}^d X_i^2 +X_0 \right)(P_{T-s} f) (Y_s) dt +\sum_{i=1}^d  \int_0^T (X_iP_{T-t}f)(Y_t)dB^i_t .
\]
But,
\[
 \left( \frac{\partial}{\partial s} + \frac{1}{2} \sum_{i=1}^d X_i^2 +X_0 \right)(P_{T-s} f)(Y_s) =-V(Y_s)(P_{T-s} f)(Y_s)
 \]
and  therefore
 \[
(P_{T-t} f) (Y_t) =P_Tf (Y_0) -\int_0^t V(Y_s)(P_{T-s} f)(Y_s) dt + \sum_{i=1}^d \int_0^t X_i (P_{T-s}f)(Y_s)dB^i_s .
\]
Assume now $f \ge 0$. In that case, the above computation shows that $(P_{T-t} f) (X_t)$ is a non negative submartingale.  Thus from   Lenglart-L\'epingle-Pratelli  (see Theorem 3.2 in \cite{LenLepPra}) and Doob's maximal inequality, we have 
\begin{align*}
\left\| P_Tf (Y_0) -\int_0^T V(Y_s)(P_{T-s} f)(Y_s) dt \right\|_p \le p  \left\| \sup_{0 \le t \le T} (P_{T-t} f) (Y_t)  
\right\|_p \le \frac{p^2}{p-1} \| f \|_p.
\end{align*}
We conclude
\[
\left\| \sum_{i=1}^d  \int_0^T (X_iP_{T-t}f)(Y_t)dB^i_t  \right\|_p \le \frac{2p^2}{p-1} \| f \|_p.
\]
For a general $f$, we can write
\[
f=f^+-f^-
\]
and we see that
\begin{eqnarray*}
\left\| \sum_{i=1}^d \int_0^t X_i (P_{T-s}f)(Y_s)dB^i_s \right\|^p_p  & = &\left\|\sum_{i=1}^d \int_0^t X_i (P_{T-s}f^+)(Y_s)dB^i_s - 
\sum_{i=1}^d \int_0^t X_i (P_{T-s}f^-)(Y_s)dB^i_s \right\|_p^p \\
 & \le & 2^{p-1} \left\| \sum_{i=1}^d \int_0^t X_i (P_{T-s}f^+)(Y_s)dB^i_s \right\|_p^p \\
 &+& 2^{p-1} \left\|\sum_{i=1}^d \int_0^t X_i (P_{T-s}f^-)(Y_s)dB^i_s  \right\|^p_p \\
 & \le& 2^{p-1}  \left(  \frac{2p^2}{p-1} \right)^p   ( \| f \|^p_p + \|f^-\|^p_p) \\
 &\le&  2^{p-1}  \left(  \frac{2p^2}{p-1} \right)^p   \| f  \|^p_p.
\end{eqnarray*}
\end{proof}

From this we obtained the following explicit bound in \eqref{Schro}.

\begin{corollary}
Let $1<p<\infty$.  For any $V\geq 0$, $f \in L_\mu^p (\mathbb{M})$,
\[
\| \mathcal{S}_A f \|_p \le 8  \| A \| (p^*-1)  \frac{p^4}{(p-1)^2}  \| f \|_p
\]
\end{corollary}

\begin{proof}  We first note that the martingale 
$$
\sum_{i,j=1}^d \int_0^T   A_{ij}(T-t)(X_j P_{T-t} f)(Y_t)dB^i_t 
$$
is differentially subordinate to 
$$
\|A\|\sum_{i=1}^d \int_0^T (X_i P_{T-t} f)(Y_t)dB^i_t. 
$$
From Theorem \ref{symm} and the above proposition  we obtain 
\begin{align*}
\int_\bM (\mathcal{S}^T_A f)g d\mu& =\sum_{i,j=1}^d \int_0^{T} \int_{\mathbb{M}}  A_{ij}(t)(X_iP_{t} g) (X_jP_{t} f) d\mu dt \\
 &= \mathbb{E} \left(  \sum_{i,j=1}^d \int_0^T   A_{ij}(T-t)(X_j P_{T-t} f)(Y_t)dB^i_t  \sum_{i=1}^d \int_0^T  (X_i P_{T-t} g)(Y_t)dB^i_t  \right) \\
 &\le \left\| \sum_{i,j=1}^d \int_0^T   A_{ij}(T-t)(X_j P_{T-t} f)(Y_t)dB^i_t  \right\|_p  \left\| \sum_{i=1}^d \int_0^T  (X_i P_{T-t} g)(Y_t)dB^i_t  \right \|_q \\
 & \le \| A \| (p^*-1)  \left\|  \sum_{i=1}^d \int_0^T  (X_i P_{T-t} f)(Y_t)dB^i_t \right\|_p  \left\|   \sum_{i=1}^d \int_0^T  (X_i P_{T-t} g)(Y_t)dB^i_t \right \|_q \\
 & \le  \| A \| (p^*-1)   \frac{2^{2-\frac{1}{p} }p^2}{p-1}  \frac{2^{2-\frac{1}{q} }q^2}{q-1} \| f \|_p \| g \|_q \\
 & \le 8  \| A \| (p^*-1)  \frac{p^4}{(p-1)^2}  \| f \|_p \| g \|_q,
\end{align*}
which implies  the corollary. 
\end{proof}

\section{Applications}

\subsection{Multipliers of Laplace transform-type for Schr\"odinger operators on manifolds}

We work in the same setting as the previous Section. We consider the following \textit{multiplier} for the Schr\"odinger operator $L=-\frac{1}{2}\sum_{i=1}^d X_i^* X_i +V$. First,  let $a\in L^{\infty}[0, \infty)$. Set 
\[
T_a f =\int_0^{+\infty}  a(t) L P_{2t} f dt.
\]

We can observe that since $L$ is essentially self-adjoint, from spectral theorem, there is a measure space $(\Omega, \nu)$, a unitary map $U: L^2_\nu (\Omega) \rightarrow L^2_\mu(\mathbb{M})$ and a non negative measurable function on $\Omega$ such that
\[
U^{-1}LU f (x) =-\lambda(x) f (x), \quad x \in \Omega.
\]
The operator $T_a$ acts on these  as multiplication operator  on $L^2_\nu (\Omega)$ in the sense that
\[
U^{-1}T_aU f (x) =-\lambda(x) \int_0^{+\infty}  a(t) e^{-t \lambda(x)} dt f(x), \quad x \in \Omega.
\]

Let us also observe that in several settings the operators $T_a$ can be interpreted as multipliers in the Fourier analysis sense. For instance, let $G$ be a compact Lie group.  Let $U$ be a bounded operator on $ L^2(G)$ which commutes with left and right translations. Then there exists a bounded function $\Phi(m)$ on $\hat{G}$, the space of equivalence classes of irreducible unitary representations of $G$, such that
\[
 Uf =\sum_{m \in \hat{G}} \Phi (m)d_m \chi_m * f, 
 \]
  where $\chi_m$  is the character and $d_m$ the dimension of the representation. Conversely, for any bounded $\Phi$, the above operator defines  a bounded operator on $ L^2(G)$ which commutes with left and right translations.  In this framework, if $L$ is an essentially self-adjoint diffusion operator that commutes with left and right translations  (like the Laplace-Beltrami operator for a bi-invariant metric), then we have
  \[
  T_af=\sum_{m \in \hat{G}} -\lambda (m)  \int_0^{+\infty}  a(t) e^{-t \lambda(m)} dt d_m \chi_m * f
  \]
where $(-\lambda(m))_{m \in \hat{G}}$ is the spectrum of $L$.

\
%
 

With the notations of the previous section, we see that
\[
T_a f =-\frac{1}{2}S_Af +\int_0^{+\infty} a(t) P_t V P_t f dt
\]
where $A(t)=a(t) \mathbf{Id}$. In the following we denote by $Q_t$ the Markovian semigroup with generator $-\frac{1}{2}\sum_{i=1}^d X_i^* X_i $.

\begin{proposition} Fix $1<p<\infty$. 
\begin{itemize}
\item[(i)] For any nonnegative potential  $V$ that satisfies $V \le -m$ for some $m \ge 0$, we have 
\begin{equation}
\| T_a f \|_p \le \left( 4  \| a \|_\infty (p^*-1)  \frac{p^4}{(p-1)^2}+\int_0^{+\infty} |a(t)| e^{-2mt} \| Q_t |V |^q \|_\infty^{1/q}  dt  \right) \| f \|_p,
\end{equation}
with $\frac{1}{q}+\frac{1}{p}=1$.
\item[(ii)] Suppose  $V\equiv 0$.  Let $a$ be such that for all $t\in [0, \infty)$, $-\infty<b\leq a(t)\leq B<\infty$.  Then 
\begin{equation}
\| S_a f \|_p \leq \frac{1}{2} C_{p, b, B}\|f\|_p,
\end{equation}
where the $C_{p, b, B}$ is the constant of Theorem \ref{non-symm}. 
\end{itemize}
\end{proposition}

\begin{proof}
Since 
\[
T_a f =-\frac{1}{2}S_Af +\int_0^{+\infty} a(t) P_t V P_t f dt,
\]
using the results of the previous section, we only need to bound in $L^p$ the operator  $\int_0^{+\infty} a(t) P_t V P_t f dt$. From Feynman-Kac formula, we have
\[
P_t V P_t f  (x) =\mathbb{E} \left( e^{\int_0^t V(X_s^x) ds} V(X_t^x) P_t f (X_t^x) \right),
\]
where $(X_t^x)_{t \ge 0}$ is the diffusion with generator  $-\frac{1}{2}\sum_{i=1}^d X_i^* X_i $ started at $x \in \mathbb{M}$. Thus
\begin{align*}
 |P_t V P_t f  (x) |  & \le \mathbb{E} \left( e^{q\int_0^t V(X_s^x) ds} |V(X_t^x)|^q \right)^{1/q}  \mathbb{E} \left( |P_t f (X_t^x)|^p  \right)^{1/p} \\
  &\le e^{-2mt} \| Q_t |V |^q \|_\infty^{1/q} \mathbb{E} \left( Q_t | f| (X_t^x)^p  \right)^{1/p} \\
  & \le e^{-2mt} \| Q_t |V |^q \|_\infty^{1/q} \left( Q_{2t} | f|^p \right) (x)^{1/p} 
\end{align*}
As a consequence,
\[
\|P_t V P_t f \|_p  \le e^{-2mt} \| Q_t |V |^q \|_\infty^{1/q} \| f \|_p,
\]
which yields the expected result.
\end{proof}

For instance, we immediately deduce from the previous proposition:
\begin{corollary}
Fix $1<p<\infty$. 
\begin{itemize}
\item[(i)] For any non-negative potential  $V$ that satisfying $ -M \le V \le -m$, for some $m, M \ge 0$, we have 
\begin{equation*}
\| T_a f \|_p \le \left( 4  \| a \|_\infty (p^*-1)  \frac{p^4}{(p-1)^2}+M \int_0^{+\infty} |a(t)| e^{-2mt}  dt  \right) \| f \|_p.
\end{equation*}
\item[(ii)] Assume that the semigroup $Q_t$ is ultracontractive.  For any nonnegative potential  $V$ that satisfies $ V \in L_\mu^q(\mathbb{M})$, we have 
\begin{equation*}
\| T_a f \|_p \le \left( 4  \| a \|_\infty (p^*-1)  \frac{p^4}{(p-1)^2}+\| V \|_q  \int_0^{+\infty} |a(t)| \| Q_t \|_{\infty ,1}  dt  \right) \| f \|_p,
\end{equation*}
with $\frac{1}{q}+\frac{1}{p}=1$.
\end{itemize}
\end{corollary}

\begin{remark} Laplace transform-type operators  have been extensively studied in many settings.  For some of this literature, see \cite{Var}, \cite{Hyt1}, \cite{Hyt2}, \cite{Ban}, and references contained in those papers. 

\end{remark}

\subsection{Second order Riesz transforms on Lie groups of compact type}\label{LieRiesz}

Let $G$ be a Lie group of compact type with Lie algebra $\mathfrak{g}$. We endow $G$ with a bi-invariant Riemannian structure and consider an orthonormal basis $X_1,\cdots, X_d$ of $\mathfrak{g}$.  In this setting the Laplace-Beltrami operator can be written as 
\[
L=\frac{1}{2} \sum_{i=1}^d X_i^2.
\]
It is essentially self-adjoint on the space of smooth and compactly supported functions and the assumptions of the previous section are satisfied. It is remarkable here that $X_i^*=-X_i$ and the vector fields $X_i$ do commute with the semigroup $P_t=e^{tL}$. In this case, if the matrix $A$ only depends on time, we get therefore
\[
\mathcal{S}_A f=-\sum_{i,j}\int_{0}^{+\infty} A_{ij} (t)   (X_i   X_j P_{2t} f ) dt.
\]
In particular, for a constant $A$, we obtain
\[
\mathcal{S}_A f=\sum_{i,j} A_{ij}  \left(\sum_{i=1}^d X_i^2\right)^{-1} X_i   X_j  f .
\]
Defining the Riesz transforms on $G$ by 
$$
R_jf=  \left(-\sum_{i=1}^d X_i^2\right)^{-1/2} X_jf
$$
we see that 
$$
\mathcal{S}_A f=\sum_{i,j=1}^d A_{ij} R_iR_j f. 
$$
We have the following result which follows from Theorems \ref{main} and \ref{non-symm-S_A}.  These inequality are exactly as those proved in $\bR^d$ for the classical Riesz transforms. 
\begin{theorem}\label{Lie-Riesz} Fix $1<p<\infty$.
\begin{itemize}
\item[(i)] For any constant coefficient matrix $A$, 
$$\|\sum_{i,j=1}^d A_{ij} R_iR_j f\|_p\leq (p^*-1)\|A\| \|f\|_p.$$
\item[(ii)] Assume that $A=(A_{ij})_{i,j=1}^d$ is symmetric matrix  with real entries and eigenvalues $\lambda_1\leq \lambda_2\leq \ldots \leq \lambda_d$. Then 
\begin{equation}\label{generalriesz}
\|\sum_{i,j=1}^d A_{ij}R_iR_jf \|_p\leq C_{p,\lambda_1,\lambda_d}\|f\|_p
\end{equation}
and this inequality is sharp.  In particular, if  $J\subsetneq \{1,2,\,\ldots,d\}$, then 
\begin{equation}\label{squareriesz}
\|\sum_{j\in J} R_j^2f\|_p\leq C_{p,0,1}\|f\|_p=c_p\|f\|_p, 
\end{equation}
where $c_p$ is the Choi constant in \eqref{choi1} and this inequality is also sharp. 
\end{itemize}
\end{theorem}
The fact that the constants in this theorem are best possible follows from the fact that they are already best possible on $\bR^d$. These results simply show that the construction in \cite{BanOse} for Riesz transforms on $\bR^d$ extend to  Lie groups of compact type without change in their norms.  We may observe that, slightly more generally,  a similar statement holds on Riemannian manifolds for which the gradient $\nabla$ commutes with the heat semigroup $P_t$. From the Bakry-\'Emery criterion (see \cite{Bakry} ), such a commutation implies that the Ricci curvature of $\mathbb{M}$ is non negative.

If $G$ is a semisimple compact Lie group (see for instance the classical reference \cite{Hel} for an account about structure theory and harmonic analysis on semisimple compact Lie groups), we can deduce from the above result an interesting class of multipliers. 
\begin{proposition}
Let $G$ be a compact semisimple Lie group. Let $\Lambda^+$ be the set of highest weights and $\Delta$ be the set of  all roots. For $\alpha, \beta \in \Delta$, denote
\[
U_{\alpha,\beta} f =\sum_{\lambda \in \Lambda^+} d_\lambda \frac{\langle \lambda,  \alpha \rangle \langle \lambda,  \beta \rangle}{\| \lambda +\rho \|^2 -\| \rho \|^2} \chi_\lambda * f,
\]
where $\rho=\frac{1}{2} \sum_{\alpha \in \Delta} \alpha$, $\chi_\lambda$ is the character of the highest weight representation and $d_\lambda$ its dimension. Then, for $1<p<\infty$, $U_{\alpha,\beta}$ is bounded in $L^p(G)$ and
\[
\| U_{\alpha,\beta} f \|_p \leq (p^*-1) \| \alpha \| \| \beta \| \|f\|_p.
\]
\end{proposition}

\begin{proof}
Let $T$ be a maximal torus of $G$ and let $\mathbf{t}_0$ be its Lie algebra. Then, the sub-algebra $\mathbf{t}$ of $\mathfrak{g}$ generated by $\mathbf{t}_0$ is a Cartan sub-algebra. Let $\mathbf{t}^*$ denote the dual of $\mathbf{t}$. If $\alpha \in \Delta$, we denote by $H_\alpha$ the element of $\mathbf{t}$ such that for every $H \in  \mathbf{t}$, $\langle H, H_\alpha\rangle= \alpha (H)$ where $\langle  \cdot,\cdot \rangle$ is induced from the Killing form. With these notations, we  see that
\[
U_{\alpha,\beta}=  C^{-1}H_\alpha H_\beta,
\]
where $C$ is the Casimir operator which is also the Laplace-Beltrami operator.
\end{proof}

\end{document}